\documentclass{amsart}
\usepackage{amsmath,amssymb,epsfig,amscd,xy,amsthm}

\usepackage{verbatim}

\numberwithin{equation}{section}

\newtheorem{lemma}[equation]{Lemma}
\newtheorem{thm}[equation]{Theorem}

\newtheorem{cor}[equation]{Corollary}
\newtheorem{prop}[equation]{Proposition}

\newtheorem{question}[equation]{Question}

\newtheorem{defi}[equation]{Definition}

\theoremstyle{remark}

\newtheorem*{remark}{Remark}
\newtheorem*{notation}{Notation}
\newtheorem*{acknowledgments}{Acknowledgments}

\newcommand{\col}{\,{:}\,}

\newcommand{\lra}{\longrightarrow}

\DeclareMathOperator{\End}{{End}}
\DeclareMathOperator{\CC}{\mathcal{C}}
\newcommand{\N}{{\mathbb N}}
\newcommand{\Z}{{\mathbb Z}}
\newcommand{\Q}{{\mathbb Q}}
\newcommand{\bC}{{\mathbb C}}
\newcommand{\bG}{{\mathbb G}}
\newcommand{\OO}{{\mathcal O}}
\newcommand{\cU}{\mathcal{U}}
\newcommand{\cG}{\mathcal{G}}
\newcommand{\cV}{\mathcal{V}}
\newcommand{\cS}{\mathcal{S}}

\newcommand{\tensor}{\otimes}
\newcommand{\ok}{\mathfrak{O}_K}
\newcommand{\cR}{\mathcal{R}}

\begin{document}



\title[The Mordell--Lang question for endomorphisms]
{The Mordell--Lang question for endomorphisms of semiabelian varieties}

\author{D.~Ghioca}
\address{
Dragos Ghioca\\
Department of Mathematics\\
University of British Columbia\\
Vancouver, BC V6T 1Z2\\
Canada
}
\email{dghioca@math.ubc.ca}

\author{T.~J.~Tucker}
\address{
Thomas Tucker\\
Department of Mathematics\\
University of Rochester\\
Rochester, NY 14627\\
USA
}
\email{ttucker@math.rochester.edu}

\author{M.~E.~Zieve}
\address{
Michael E. Zieve\\
Department of Mathematics\\
University of Michigan\\
530 Church Street\\
Ann Arbor, MI 48109\\
USA
}
\email{zieve@umich.edu}

\thanks{The first author was partially supported by NSERC. The second author was partially 
supported by NSF Grants 0801072 and 0854839.}


\begin{abstract}
The Mordell--Lang conjecture describes the intersection of a finitely generated
subgroup with a closed subvariety of a semiabelian variety.  Equivalently, this
conjecture describes the intersection of closed subvarieties with the set of
images of the origin under a finitely generated semigroup of translations.
We study the analogous question in which the translations are replaced by
algebraic group endomorphisms (and the origin is replaced by another point).
We show that the conclusion of the Mordell--Lang conjecture
remains true in this setting if either (1) the semiabelian variety is simple, (2) the semiabelian variety is $A^2$, where $A$ is a one-dimensional semiabelian variety,
(3) the subvariety is a connected one-dimensional
algebraic subgroup, 
or (4) each endomorphism has diagonalizable Jacobian at
the origin.  We also give examples showing that the conclusion fails if we
make slight modifications to any of these hypotheses.

La conjecture du Mordell-Lang d\'{e}crit l'intersection d'une sous-groupe de type fini avec une vari\'{e}t\'{e} ferm\'{e}e d'une vari\'{e}t\'{e} semi-ab\'{e}lienne. Equivalemment, cette conjecture d\'{e}crit l'intersection de sous-vari\'{e}t\'{e}s ferm\'{e}es avec la collection d'images de l'origine sous un semigroupe de translat\'es de type fini. Nous \'{e}tudions la question analogue dans laquelle les translat\'es sont
remplac\'{e}es par les endomorphismes de groupe alg\'{e}briques (et l'origine est remplac\'{e}e par un autre point). Nous montrons qui la conclusion de la  conjecture du Mordell-Lang reste vraie dans ce param\`{e}tre si ou (1) la vari\'{e}t\'{e} semi-ab\'{e}lienne est simple, (2) la vari\'{e}t\'{e} semi-ab\'{e}lienne est $A^2$, o\`{u} $A$ est une vari\'{e}t\'{e} semi-ab\'{e}liennne de dimension $1$, (3) la sous-vari\'{e}t\'{e} est une sous-vari\'{e}t\'{e} semi-ab\'{e}liennne de dimension $1$, ou (4) la matrice jacobienne \`{a} l'origine du chaque endomorphisme est diagonalisable. Nous donnons aussi des exemples qui montre que la conclusion \'{e}choue si nous faisons l'affront \`{a} n'importe lequel de ces hypoth\'{e}ses. 
\end{abstract}

\date{\today}

\maketitle


\section{Introduction}

Lang's generalization of the Mordell conjecture describes the intersection
of certain subgroups and subvarieties of algebraic groups.  Specifically,
this conjecture addresses \emph{semiabelian varieties}, which are the
connected algebraic groups $G$ admitting an exact sequence
\[ 1 \to \bG_m^k \to G \to A \to 1 \]
with $A$ an abelian variety and $\bG_m$ the multiplicative (algebraic) group.
The Mordell--Lang conjecture, proved by Faltings \cite{Faltings} for
abelian varieties and by Vojta \cite{V1} in general, is as follows:

\begin{thm}
\label{T:F}
Let $G$ be a semiabelian variety defined over\/ $\bC$.
Let $V$ be a closed subvariety of $G$, and let\/
$\Gamma$ be a finitely generated subgroup of
$G(\bC)$.  Then $V(\bC)\cap\Gamma$ is the union of finitely many
cosets of subgroups of\/ $\Gamma$.
\end{thm}

McQuillan \cite{McQuillan} extended Vojta's result to the case of finite rank subgroups $\Gamma\subset G(\bC)$, i.e. when $\dim_{\Q}\Gamma\tensor_{\Z}\Q$ is finite.

We can view Theorem~\ref{T:F} as a dynamical assertion by writing
$\Gamma$ as the set of images of the origin under translations by elements
of $\Gamma$.  We will study the analogous problem
in which the translation maps are replaced by more general endomorphisms
of $G$, and the origin is replaced by an arbitrary point of $G(\bC)$.
In this setting, the analogue of $\Gamma$ is not generally a group,
so we must reformulate the conclusion of Theorem~\ref{T:F}, as follows:

\begin{question}
\label{semigroup semiabelian}
Let $G$ be a semiabelian variety defined over\/ $\bC$, let\/
$\Phi_1,\dots,\Phi_r$
be commuting endomorphisms of $G$, let $V$ be a closed subvariety of $G$, and
let $\alpha\in G(\bC)$.  Let $E$ be the set of tuples
$(n_1,\dots,n_r)\in\N^r$ for which\/
$\Phi_1^{n_1}\cdots\Phi_r^{n_r}(\alpha)\in V(\bC)$.  Is $E$ the union of
finitely many sets of the form $u+(\N^r\cap H)$ with $u\in\N^r$ and $H$ a
subgroup of\/~$\Z^r$?
\end{question}

Theorem~\ref{T:F} implies that the answer is `yes' if each $\Phi_i$ is a
translation (even if $\alpha\ne 0$); conversely, it is not difficult to
deduce Theorem~\ref{T:F} as a consequence of
Question~\ref{semigroup semiabelian} in case the $\Phi_i$ are translations.
The crucial fact here is that the group of points on a semiabelian variety
is commutative \cite[Lemma~4]{iitaka2}.  On a related note, one could consider
Question~\ref{semigroup semiabelian} for noncommuting endomorphisms, but
we know of almost no results in this setting.

Each endomorphism of a semiabelian variety is the composition of a translation
and an algebraic group endomorphism \cite[Thm.~2]{iitaka}.
Question~\ref{semigroup semiabelian} generally has
a negative answer if the $\Phi_i$'s include both translations and
algebraic group endomorphisms.  For instance, let $P_0$ be a nontorsion point
on a semiabelian variety $G_0$, and let $V$ be the diagonal in $G:=G_0^2$; for
\[
\Phi_1\colon (x,y)\mapsto (x+P_0,y) \quad\text{ and \quad}
\Phi_2\colon (x,y)\mapsto (x,2y),
\]
if $\alpha=(0,P_0)$ then
$E=\{(2^k,k): k\in\N_0\}$ does not have the desired form.
However, we will show that Question~\ref{semigroup semiabelian} often has an
affirmative answer when every $\Phi_i$ is an algebraic group endomorphism.

\begin{thm}
\label{main}
Let $G$ be a semiabelian variety defined over $\bC$, let $\alpha\in G(\bC)$,
let $V\subset G$ be a closed subvariety, and let $\Phi_1,\dots,\Phi_r$ be
commuting algebraic group endomorphisms of $G$.  Assume that either
\begin{enumerate}
\item[(a)] The Jacobian at $0$ of each endomorphism $\Phi_i$ is diagonalizable; or
\item[(b)] $V$ is a connected algebraic subgroup of $G$ of dimension one; or
\item[(c)] $G=A^j$ for $0\le j\le 2$, where $A$ is a one-dimensional semiabelian variety.
\end{enumerate}
Then the set $E$ of tuples $(n_1,\dots,n_r)\in\N^r$ for which
$\Phi_1^{n_1}\dots\Phi_r^{n_r}(\alpha)\in V(\bC)$ is the union of
finitely many sets of the form $u+(\N^r\cap H)$ with $u\in\N^r$ and $H$ a
subgroup of\/~$\Z^r$.
\end{thm}

It is immediate to see that if condition (a) holds in Theorem~\ref{main}, then the Jacobian at $0$ of \emph{each} endomorphism of the semigroup $S$ generated by the $\Phi_i$'s is diagonalizable. 

The first two authors showed in \cite{p-adic} that the conclusion
of Theorem~\ref{main} holds when $r=1$, even if one does not assume (a), (b), or (c).
However, in Section~\ref{counterexamples} we will give examples with $r=2$ and
$G=\bG_m^3$ where the conclusion of Theorem~\ref{main} does not hold.  Our examples
lie just outside the territory covered by Theorem~\ref{main}, as $V$ is a coset of
a one-dimensional algebraic subgroup of $G$ in the first example, and
$V$ is a two-dimensional algebraic subgroup of $G$ in the second example.

In our previous paper \cite{lines-2} we discussed a more general version of
Question~\ref{semigroup semiabelian} involving endomorphisms of an arbitrary
variety.  
\begin{question}
\label{general semigroup question}
Let $X$ be any quasiprojective variety defined over\/ $\bC$, let\/
$\Phi_1,\dots,\Phi_r$
be commuting endomorphisms of $X$, let $V$ be a closed subvariety of $G$, and
let $\alpha\in X(\bC)$.  Let $E$ be the set of tuples
$(n_1,\dots,n_r)\in\N^r$ for which\/
$\Phi_1^{n_1}\cdots\Phi_r^{n_r}(\alpha)\in V(\bC)$.  Is $E$ the union of
finitely many sets of the form $u+(\N^r\cap H)$ with $u\in\N^r$ and $H$ a
subgroup of\/~$\Z^r$?
\end{question}

The only known instances where Question~\ref{general semigroup question} has an
affirmative answer are the results of the present paper, the results of
our previous paper \cite{lines-2}, the Mordell--Lang conjecture itself,
and various cases with $r=1$
\cite{Skolem,Mahler,Lech,Denis,Bell,lines,PR,p-adic,Jason}.

Here is a rough sketch of our proof for Theorem~\ref{main}. At the
expense of replacing $V$ by the Zariski closure of the intersection
$V(\bC)\cap \OO_S(\alpha)$, we may assume that $\OO_S(\alpha)$ is
Zariski dense in $V$. In particular, letting $\Gamma$ be the finitely
generated subgroup
of $G(\bC)$ generated by $\OO_S(\alpha)$, we obtain that
$V(\bC)\cap\Gamma$ is Zariski dense in $V$. So, assuming $V$ is
irreducible, Theorem~\ref{T:F} yields that $V$ is a coset of an
algebraic subgroup of $G$. We may choose a suitable prime number $p$
and then find a suitable model of the semiabelian variety $G$ over
$\Z_p$; finally, using the $p$-adic exponential map on $G$, we reduce
the problem to a question of linear algebra inside $\bC_p^g$ (where
$\dim(G)=g$). Given finitely many commuting matrices $J_i\in
M_{g,g}(\bC_p)$ (for $i=1,\dots,r$), given $u_0,v_0\in \bC_p^g$, and
given a linear subvariety $\cV\subset \bC_p^g$, we need to describe
the set $E$ of tuples $(n_1,\dots,n_r)\in\N^r$ such that
$J_1^{n_1}\cdots J_r^{n_r}(u_0)\in \left(v_0 + \cV\right)$. We show
that $E$ has the description from the conclusion of Theorem~\ref{main}
if either (a) each $J_i$ is diagonalizable, or (b) $\dim(\cV)=1$ and
$v_0$ is the zero vector, or (c) $g=2$ and the matrices $J_i$ are all
defined over the ring of integers in a quadratic imaginary field.

The contents of this paper are as follows.
In the next section we prove some basic results regarding semigroups of $\N^r$ which are used throughout our proofs. As a consequence of our results from Section~\ref{results on semigroups}, we prove in Proposition~\ref{zero-dimensional} that Question~\ref{general semigroup question} has a positive answer in greater generality assuming $\dim(V)=0$. Then in Section~\ref{reductions} we prove several preliminary results towards proving 
Theorem~\ref{main}, while in Section~\ref{p-adic section} we reduce the proof of our results to propositions on linear algebra by using suitable $p$-adic parametrizations. Then in Section~\ref{our main proofs} we prove Theorem~\ref{main}.
Finally, in Section~\ref{counterexamples} we provide counterexamples to our theorem if we
relax its hypothesis. 

\begin{notation}
In this paper, each subvariety is closed, and each semigroup contains an
identity element.  We also denote by $\N$ the set of all nonnegative integers.

For any algebraic group $G$, we denote by $\End(G)$ the set of all algebraic
group endomorphisms of $G$.

For any set $T$, for any set $S$ of maps $\Psi\colon T\to T$, and for any
$\alpha\in T$, we let $\OO_S(\alpha):=\{\Psi(\alpha)\text{ : }\Psi\in S\}$ be
the $S$-orbit of $\alpha$.

We will use the following definition: if $\Phi$ is an endomorphism of the quasiprojective variety $X$, and if $V$ is a subvariety of $X$, then we say that $V$ is \emph{preperiodic} under $\Phi$ if there exist $k,\ell\in \N_0$ with $k< \ell$ such that $\Phi^k(V)=\Phi^{\ell}(V)$. If we may take $k=0$ in the above definition, then we say that $V$ is \emph{periodic} under $\Phi$.
\end{notation}

\begin{acknowledgments}
We would like to thank Pascal Autissier, Cristin Kenney, Fabien
Pazuki, and Thomas Scanlon for many helpful conversations, and we also thank the referee for several useful suggestions. We would
also like to thank the Centro di Ricerca Matematica Ennio De
Giorgi and the University of Bordeaux 1 for their hopsitality.
\end{acknowledgments}
 

\section{Semigroups of $\N^r$}
\label{results on semigroups}

In this section we provide basic results regarding subsets of $\N^r$. We start with a definition.
\begin{defi}
\label{the class}
We denote by $\CC$ the class of all sets of the form
$$\bigcup_{i=1}^{\ell} \left(\gamma_i + H_i\cap \N^r\right),$$
for some nonnegative integers $\ell$ and $r$, and for some $\gamma_i\in \N^r$, and for some subgroups $H_i\subset \Z^r$.
\end{defi}

In our proofs we will use the following order relation on $\N^r$.
\begin{defi}
We define the order $\prec_r$ on $\N^r$ as follows:
$$(\beta_1,\dots,\beta_r)\prec_r (\gamma_1,\dots,\gamma_r)$$
if and only if $\beta_i\le \gamma_i$ for all $i=1,\dots,r$. Equivalently, $\beta \prec_r \gamma$ as elements of $\Z^r$ if and only if $\gamma-\beta\in \N^r$. When it is clear in the context, we will drop the index $r$ from $\prec_r$ and simply denote it by $\prec$.

Additionally, for each subset $S\subset \N^r$, we call $\gamma\in S$ a minimal element (with respect to the above ordering) if there exists no $\beta\in S\setminus\{\gamma\}$ such that $\beta\prec \gamma$. 
\end{defi}

Our first result will be used many times throughout our paper, and it is an easy consequence of the order relation $\prec_r$ on $\N^r$.
\begin{prop}
\label{finitely many minimal elements}
Any subset $S\subset \N^r$ has at most finitely many minimal elements with respect to $\prec_r$.
\end{prop}

\begin{proof}
Assume $S$ is nonempty (otherwise the statement is vacuously true). We prove the result by induction on $r$. The case $r=1$ is obvious.

Now, assume the result is proved for $r\le k$ and we prove it for $r=k+1$ (for some $k\in\N$). Let $\pi$ be the natural projection of $\N^{k+1}$ on its last $k$ coordinates, and let $S_1=\pi(S)$. Let $M_1$ be the set of minimal elements for $S_1$ with respect to the ordering $\prec_k$ of $\N^k$ defined on the last $k$ coordinates of $\N^{k+1}$; according to the induction hypothesis, $M_1$ is a finite set. For each $\overline{\gamma_1}\in M_1$, we let $\overline{\gamma}\in S$ be the minimal element of $\pi^{-1}(\overline{\gamma_1})\cap S$ with respect to the ordering $\prec_{k+1}$ on $\N^{k+1}$. Obviously, $\overline{\gamma}=(\gamma_1,\overline{\gamma_1})\in \N\times \N^k$ with $\gamma_1\in \N$ minimal such that $\pi(\overline{\gamma})=\overline{\gamma_1}$; we let $M_0$ be the finite set of all these $\overline{\gamma}$.

Let $\ell$ be the maximum of all $\gamma_1\in \N$ such that there exists $\overline{\gamma_1}\in M_1$ for which $(\gamma_1,\overline{\gamma_1})\in M_0$. Then for each $i=0,\dots, \ell-1$, we let $S_{1,i} =\pi\left(S\cap \left(\{i\}\times \N^k\right)\right)$. Again using the induction hypothesis, there exists at most a finite set $M_{1,i}$ of minimal elements of $S_{1,i}$ with respect to $\prec_k$. We let $M_{0,i}$ be the set (possibly empty) given by $M_{0,i}:=\left(\{i\}\times \N^{k}\right)\cap \pi^{-1}\left(M_{1,i}\right)$, and we let $M:=M_0\bigcup \left(\bigcup_{i=0}^{\ell-1} M_{0,i}\right)$. Then $M$ is a finite set, and we claim that it contains all minimal elements of $S$. Indeed, if $(\beta_1,\dots,\beta_{k+1})\in S\setminus M$ is a minimal element with respect to $\prec_{k+1}$, then there are two possibilities.

{\bf Case 1.} $(\beta_2,\dots,\beta_{k+1})\in M_1$, i.e. $(\beta_2,\dots,\beta_{k+1})$ is a minimal element of $S_1$.

In this case, since $(\beta_1,\dots,\beta_{k+1})\notin M_0$, we get (from the definition of $M_0$) that there exists $j\in \N$ such that $j<\beta_1$ and $(j,\beta_2,\dots,\beta_{k+1})\in M_0$ contradicting thus the minimality of $(\beta_1,\dots,\beta_{k+1})$ in $S$.

{\bf Case 2.} $(\beta_2,\dots,\beta_{k+1})\notin M_1$, i.e. $(\beta_2,\dots,\beta_{k+1})$ is not a minimal element of $S_1$.

In this case, there exists $(\gamma_1,\overline{\gamma_1})\in M_0$ such that $\overline{\gamma_1}\prec_k (\beta_2,\dots,\beta_{k+1})$. Now, since we assumed that $(\beta_1,\dots,\beta_{k+1})$ is a minimal element of $S$, we obtain that $\beta_1 < \gamma_1\le \ell$. Since we assumed that $(\beta_1,\dots,\beta_{k+1})\notin M_{0,\beta_1}$, then there exists $(\beta_1,\dots,\beta_{k+1})\ne (\beta_1,\overline{\beta_1})\in M_{0,\beta_1}$ such that $\overline{\beta_1}\prec_k (\beta_2,\dots,\beta_{k+1})$, and thus $(\beta_1,\overline{\beta_1})\prec_{k+1} (\beta_1,\beta_2\dots,\beta_{k+1})$ contradicting our assumption on the minimality of $(\beta_1,\dots,\beta_{k+1})$.

Therefore $M$ is a finite set containing all minimal elements of $S$.
\end{proof}

In particular, Proposition~\ref{finitely many minimal elements} yields the following result.
\begin{prop}
\label{finitely generated semigroups}
Let $H\subset \Z^r$ be a subgroup. Then $H\cap\N^r$ is a finitely generated semigroup of $\N^r$.
\end{prop}

\begin{proof}
Assume $S:=H\cap \N^r$ is not the trivial semigroup (in which case our claim is vacuously true). Hence, let $M$ be the finite set of all minimal elements with respect to $\prec$ of $S\setminus (0,\dots,0)$. We claim that $S$ is the semigroup generated by $M$. 

Indeed, assume the semigroup $S_0$ generated by $M$ is not equal to $S$. Clearly $S_0\subset S$ since $S$ is the intersection of $\N^r$ with a subgroup of $\Z^r$. Let $\gamma$ be a minimal element with respect to $\prec$ of $S\setminus S_0$. Since $(0,\dots,0)\ne \gamma\notin M$, we conclude that there exists $\gamma\ne \beta\in M\subset S_0$ such that $\beta\prec \gamma$. But then $\gamma-\beta\in \N^r$ and since both $\beta$ and $\gamma$ are also in $H$, we get that $\gamma-\beta\in S$. Since $\gamma\notin S_0$, then also $\gamma-\beta\notin S_0$, but on the other hand $\gamma \ne(\gamma - \beta) \prec \gamma$ contradicting thus the minimality of $\gamma$.

This concludes our proof that $S$ is generated as a subsemigroup of $\N^r$ by the finite set $M$.
\end{proof}

The following propositions yield additional properties of subsemigroups of $\N^r$.
\begin{prop}
\label{cosets of subgroups yield subsemigroups}
For any subgroup $H\subset \Z^r$, and for any $\gamma\in \Z^r$, the intersection $(\gamma+H)\cap \N^r$ is a union of at most finitely many cosets of subsemigroups of the form $\beta + (H\cap \N^r)$.
\end{prop}

\begin{proof}
We let $M$ be the finite set (see Proposition~\ref{finitely many minimal elements}) of minimal elements with respect to $\prec$ for the set $(\gamma+H)\cap \N^r$. We claim that
$$(\gamma+H)\cap \N^r = \bigcup_{\beta \in M} \beta + (H\cap \N^r).$$ 
Indeed, the reverse inclusion is immediate. Now, for the direct inclusion, if $\tau\in (\gamma+H)\cap \N^r$, then there exists $\beta\in M$ such that $\beta\prec \tau$. So, $\tau-\beta\in \N^r$. On the other hand, since both $\beta$ and $\tau$ are in $\gamma+H$, we also get that $\tau-\beta\in H$, as desired. 
\end{proof}

\begin{prop}
\label{intersection of cosets yields cosets}
For any $\gamma_1,\gamma_2\in\N^r$ and for any subgroups $H_1,H_2\subset \Z^r$, the intersection $\left(\gamma_1 + (H_1\cap \N^r)\right)\cap \left(\gamma_2 + (H_2\cap \N^r)\right)$ is a union of at most finitely many cosets of subsemigroups of the form $\beta + (H\cap \N^r)$, where $H:=H_1\cap H_2$.
\end{prop}

\begin{proof}
We let $M$ be the finite set (see Proposition~\ref{finitely many minimal elements}) of minimal elements with respect to $\prec$ for the set $\left(\gamma_1+(H_1\cap \N^r)\right)\cap \left(\gamma_2 + (H_2\cap \N^r)\right)$. We claim that
$$\left(\gamma_1+(H_1\cap \N^r)\right)\cap \left(\gamma_2+(H_2\cap \N^r)\right) = \bigcup_{\beta \in M} \beta + (H\cap \N^r).$$ 
Indeed, the reverse inclusion is immediate. Now, for the direct inclusion, if $\tau\in \left(\gamma_1+(H_1\cap \N^r)\right)\cap \left(\gamma_2+(H_2\cap \N^r)\right)$, then there exists $\beta\in M$ such that $\beta\prec \tau$. So, $\tau-\beta\in \N^r$. On the other hand, since both $\beta$ and $\tau$ are in $\gamma_1+H_1$ and in $\gamma_2+H_2$, we also get that $\tau-\beta\in H_1\cap H_2=H$, as desired. 
\end{proof}

\begin{prop}
\label{important step semigroup}
Let $\gamma\in \N^r$, let $H\subset \Z^r$ be a subgroup, and let $N\in\N$ be a positive integer. Assume that for a subset $L\subset \left(\gamma + H\cap \N^r\right)$, we have that for each $\tau\in H\cap \N^r$, and for each $\beta\in L$, then also $\beta +N\cdot \tau\in L$. Then $L$ is a union of at most finitely many cosets of subsemigroups of the form $\delta + \left((N\cdot H)\cap \N^r\right)$.
\end{prop}

\begin{proof}
Since $H/N\cdot H$ is finite, we may write $H$ as a finite union $\bigcup_{i=1}^s \tau_i + N\cdot H$ for some $\tau_i\in \Z^r$. Using Proposition~\ref{cosets of subgroups yield subsemigroups}, we conclude that each 
$$\gamma + \left(\tau_i + N\cdot H\right)\cap \N^r$$
is itself in the class $\CC$ (see Definition~\ref{the class}); more precisely we may write
$$\gamma + H\cap \N^r = \bigcup_{i=1}^s \left(\gamma + \left(\tau_i+N\cdot H\right)\cap \N^r\right) = \bigcup_{\beta\in M} \left(\beta+\left((N\cdot H)\cap \N^r\right)\right),$$
for a suitable finite subset $M$ of $\N^r$.
Therefore, it suffices to show that for each $\beta\in M$, the set 
$$L_{\beta}:=L\cap  \left(\beta+\left((N\cdot H)\cap \N^r\right)\right)$$
satisfies the conclusion of our Proposition~\ref{important step semigroup}. Now, for each such $\beta\in M$, we let $M_{\beta}$ be the finite set of minimal elements of $L_{\beta}$ with respect to the ordering $\prec$ of $\N^r$. Then, according to our hypothesis, we do have that
$$\bigcup_{\delta\in M_{\beta}} \left(\delta + \left((N\cdot H)\cap \N^r\right)\right)\subset L_{\beta}.$$
We claim that the above inclusion is in fact an equality. Indeed, for any $\zeta\in L_{\beta}$, there exists $\delta\in M_{\beta}$ such that $\delta\prec \zeta$. Therefore $\zeta-\delta\in \N^r$. On the other hand, both $\zeta$ and $\delta$ are in $\beta + N\cdot H$, which means that $\zeta-\delta\in N\cdot H$, and thus
$\zeta-\delta\in (N\cdot H)\cap \N^r$, as desired.
\end{proof}

\begin{prop}
\label{the cosets class is closed under projections}
Let $\pi_1:\N^r\lra \N$ be the projection on the first coordinate, and let $\pi_{2,\dots,r}:\N^r\lra \N^{r-1}$ be the projection on the last $(r-1)$ coordinates. Then for each $a\in \N$, for each $\gamma\in \N^r$ and for each subgroup $H\subset \Z^r$, we have that $\pi_{2,\dots,r}\left(\pi_1^{-1}(a)\cap (\gamma + H\cap \N^r)\right)$ is in the class $\CC$.
\end{prop}

\begin{proof}
Using Propositions~\ref{cosets of subgroups yield subsemigroups} and \ref{intersection of cosets yields cosets} we may assume that $\gamma +(H\cap \N^r)\subset \pi_1^{-1}(a)$ since $\pi_1^{-1}(a)$ is itself in the class $\CC$, and thus its intersection with $\gamma + H\cap \N^r$ is again in the class $\CC$. The conclusion of Proposition~\ref{the cosets class is closed under projections} follows since the projection of any subgroup is also a subgroup. 
\end{proof}

\begin{prop}
\label{reduction of r}
Let $L\subset \N^r$, and let $\beta:=(\beta_1,\dots,\beta_r)\in L$. For each $i=1,\dots,r$ and for each $j=0,\dots,\beta_i-1$, we let
$$L_{i,j}:=\{(\gamma_1,\dots,\gamma_r)\in L\text{ : }\gamma_i=j\}
\quad \text{and} \quad L_{\beta}:=\{\gamma\in L\text{ : }\beta\prec \gamma\}.$$
Then 
\begin{equation}
\label{decomposition of the L}
L=L_{\beta}\bigcup \left(\bigcup_{i,j}L_{i,j}\right).
\end{equation}
Moreover, if $L$ is in the class $\CC$, then each $L_{i,j}$ and $L_{\beta}$ are in the class $\CC$.
\end{prop}
\begin{proof}
The decomposition of $L$ as in \eqref{decomposition of the L} follows from the definition of the ordering $\prec$ on $\N^r$. The ``moreover'' claim follows from Proposition~\ref{the cosets class is closed under projections} -- note that $L_{\beta}\in\CC$ because 
$$L_{\beta}=L\cap\left(\beta +\N^r\right),$$
and according to Proposition~\ref{intersection of cosets yields cosets}, the class $\CC$ is closed under taking the intersection of any two elements of it.
\end{proof}

Using the results of this Section we can prove that Question~\ref{general semigroup question} \emph{always} has a positive answer when $V$ is zero-dimensional. This result will be used in the inductive hypothesis from the proof of Lemma~\ref{replace by a multiple}.
\begin{prop}
\label{zero-dimensional}
Let $X$ be any quasiprojective variety defined over an arbitrary field $K$, let $\Phi_1,\dots,\Phi_r$ be any set of commuting endomorphisms of $X$, let $\alpha\in X(K)$, and let $V$ be a zero-dimensional subvariety of $X$. Then the set of tuples $(n_1,\dots,n_r)\in\N^r$ for which $\Phi_1^{n_1}\cdots \Phi_r^{n_r}(\alpha)\in V$ belongs to the class $\CC$.
\end{prop}

\begin{proof}
Obviously, it suffices to prove our result assuming $V$ consists of a single point, call it $\beta$. We let $E$ be the set of all tuples $(n_1,\dots,n_r)\in \N^r$ such that $\Phi_1^{n_1}\cdots \Phi_r^{n_r}(\alpha)=\beta$. According to Proposition~\ref{finitely many minimal elements}, the set $E$ has a finite set $M$ of minimal elements with respect to the ordering $\prec$ on $\N^r$. Now, let 
$$L:=\{(n_1,\dots,n_r)\in \N^r\text{ : }\Phi_1^{n_1}\cdots \Phi_r^{n_r}(\beta)=\beta\}.$$
Let $H$ be the subgroup of $\Z^r$ generated by $L$. We claim that $L=H\cap \N^r$. Indeed, we see from the definition of $L$ that it is closed under addition, and moreover, for any $\beta,\gamma\in L$ such that $\beta \prec \gamma$, then also $\gamma-\beta\in L$. Now it is immediate to see that
$E=\bigcup_{\tau\in M}(\tau + L)$, as desired.
\end{proof}

The following result is an immediate corollary of Proposition~\ref{zero-dimensional} and Theorem~\ref{T:F}, which shows that Question~\ref{semigroup semiabelian} has positive answer assuming $G$ is a \emph{simple} semiabelian variety (i.e. $G$ has no connected algebraic subgroups other than itself, and the trivial algebraic group).
\begin{cor}
\label{the case of simple semiabelian varieties}
Let $G$ be a simple semiabelian variety defined over $\bC$, let $\Phi_1,\dots,\Phi_r$ be any set of commuting endomorphisms of $G$, let $\alpha\in G(\bC)$, and let $V$ be any subvariety of $G$. Then the set of tuples $(n_1,\dots,n_r)\in\N^r$ for which $\Phi_1^{n_1}\cdots \Phi_r^{n_r}(\alpha)\in V$ belongs to the class $\CC$.
\end{cor}

We will use the following easy lemma both in the proof of Corollary~\ref{the case of simple semiabelian varieties}, and also in the proof of our main result Theorem~\ref{main}.

\begin{lemma}
\label{simple lemma}
Let $G$ be a semiabelian variety defined over $\bC$, let $S$ be a commutative semigroup of $\End(G)$ generated by the endomorphisms $\Phi_1,\dots,\Phi_r$ of $G$, and let $\alpha\in G(\bC)$. Then the orbit $\OO_S(\alpha)$ spans a finitely generated subgroup $\Gamma$ of $G(\bC)$.
\end{lemma}

\begin{proof}
  Each endomorphism $\Psi$ of $G$ satisfies a monic equation of degree
  $2k + l$ over $\Z$ (seen as a subring of $\End(G)$), where $1 \to
  \bG_m^k \to G \to A \to 1$ is exact and $l = \dim A$.  This follows
  by taking the characteristic polynomial of the action $\Psi$ induces
  on $H_1(G, \Z)$, which has dimension equal to $2k + l$ (see
  \cite[I.10]{Milne}, for example).  Hence $\OO_S(\alpha)$ is
  contained in the subgroup $\Gamma\subset G(\bC)$ spanned by all
  $\Phi_1^{n_1}\dots\Phi_r^{n_r}(\alpha)$ where $0\le n_i<2k + l$ for
  each $i=1,\dots,r$.
\end{proof}

\begin{proof}[Proof of Corollary~\ref{the case of simple semiabelian varieties}]
It is clear that it suffices to prove Corollary~\ref{the case of simple semiabelian varieties} assuming $V$ is irreducible. Furthermore, we may replace $V$ by the Zariski closure of $V(\bC)\cap\OO_S(\alpha)$.

If there are only finitely many distinct points in the intersection of $V$ with the orbit $\OO_S(\alpha)$ under the semigroup $S$ generated by $\Phi_1,\dots,\Phi_r$, then we are done by applying Proposition~\ref{zero-dimensional} to each of the finitely many points in $V(\bC)\cap\OO_S(\alpha)$.

Now, assume $V(\bC)\cap\OO_S(\alpha)$ is infinite. Then also $V(\bC)\cap\Gamma$ is infinite, where $\Gamma$ is the subgroup of $G(\bC)$ generated by $\OO_S(\alpha)$ (by Lemma~\ref{simple lemma}, we know that $\Gamma$ is a finitely generated subgroup). Applying Theorem~\ref{T:F}, we conclude that the irreducible subvariety $V$ is the Zariski closure of a coset of an infinite subgroup of $\Gamma$; therefore, we conclude that $V$ is a translate of a connected algebraic subgroup of $G$. However, since $G$ is a simple semiabelian variety, and since $V$ contains infinitely many points (because $V(\bC)\cap\OO_S(\alpha)$ is infinite), we conclude that $V=G$. Then the conclusion of Corollary~\ref{the case of simple semiabelian varieties} is immediate.  
\end{proof}

Corollary~\ref{the case of simple semiabelian varieties} immediately yields Theorem~\ref{main}(c) for $j=0,1$ (note that each one-dimensional semiabelian variety is simple).
\begin{cor}
\label{j=0,1}
Let $A$ be a simple semiabelian variety, let $j\in\{0,1\}$, and let $G=A^j$. Let $\Phi_1,\dots,\Phi_r$ be any set of commuting endomorphisms of $G$, let $\alpha\in G(\bC)$, and let $V$ be any subvariety of $G$. Then the set of tuples $(n_1,\dots,n_r)\in\N^r$ for which $\Phi_1^{n_1}\cdots \Phi_r^{n_r}(\alpha)\in V$ belongs to the class $\CC$.
\end{cor}


\section{Some reductions}
\label{reductions}
In this section, we will show that it suffices to prove
Theorem~\ref{main} in the case where the $\Phi_i$ are all finite maps and
where $\alpha$ has been replaced by a suitable multiple $m \alpha$.
This will allow us in Section~\ref{p-adic section} to use $p$-adic logarithms and exponentials to
translate our problem into the language of linear algebra.  

\begin{lemma}
\label{reduction to finite maps}
It suffices to prove Theorem~\ref{main} if each
$\Phi_i$ is finite.
\end{lemma}

\begin{proof}[Proof of Lemma~\ref{reduction to finite maps}.]
We proceed by induction on the number $r$ of generators of the semigroup $S$; the case $r=0$ is trivial since then the semigroup $S$ is trivial.

Assume $r=1$, i.e. $S=\langle\Phi_1\rangle$. Because
$\Phi_1^n(G)\subset \Phi_1^{n-1}(G)$ for each $n\in\N$, and because there is no
infinite descending chain of semiabelian varieties, we conclude that there
exists $N_1\in\N$ such that $\Phi_1$ is a finite endomorphism of the
semiabelian variety $G_1:=\Phi_1^{N_1}(G)$. Hence, replacing $G$ by $G_1$,
replacing $V$ by $V\cap G_1$, and replacing $\alpha$ by $\Phi_1^{N_1}(\alpha)$
we are done.

In the above argument, note that if we are in the case of Theorem~\ref{main}(c) and $G=A^2$, where $A$ is a one-dimensional semiabelian variety, in the above descending chain of semiabelian varieties one may only find the trivial group, or the group $G$ itself, or some one-dimensional semiabelian variety. Indeed, $\dim(G)=2$, and thus, each proper semiabelian subvariety of $G$ is one-dimensional. Therefore, in the argument from the above paragraph, if $G$ satisfies the hypothesis of Theorem~\ref{main}(c), then also each semiabelian subvariety $\Phi_1^i(G)$ would also satisfy the hypothesis of Theorem~\ref{main}(c).

Assume that we proved Lemma~\ref{reduction to finite maps} for all semigroups
$S$ generated by fewer than $r$ endomorphisms. Reasoning as before, for each
$i=1,\dots,r$ we find $N_i\in\N$ such that $\Phi_i$ is a finite endomorphism
of the semiabelian variety $G_i:=\Phi_i^{N_i}(G)$; let
$H:=\Phi_1^{N_1}\cdots\Phi_r^{N_r}(G)$. Then Proposition~\ref{reduction of r} applied to the element $(N_1,\dots,N_r)$ of $\N^r$ allows us to apply the inductive hypothesis. More precisely, for each $i=1,\dots,r$, we let $S_i$ be the
commutative semigroup generated by all $\Phi_j$ for $j\ne i$; then
\begin{eqnarray*}
\OO_S(\alpha) & = & \OO_S(\Phi_1^{N_1}\cdots\Phi_r^{N_r}(\alpha))\\
& \bigcup & \left(\OO_{S_1}(\alpha)\cup\OO_{S_1}(\Phi_1(\alpha))\cup \dots
  \cup\OO_{S_1}(\Phi_1^{N_1-1}(\alpha))\right)\\
& \bigcup & \left(\OO_{S_2}(\alpha)\cup\OO_{S_2}(\Phi_2(\alpha))\cup \dots
  \cup\OO_{S_2}(\Phi_2^{N_2-1}(\alpha))\right)\\
& \bigcup & \cdots \cdots \cdots \cdots \cdots \cdots \cdots \cdots \cdots
  \cdots \cdots \cdots \cdots \\
& \bigcup & \left(\OO_{S_r}(\alpha)\cup\OO_{S_r}(\Phi_r(\alpha))\cup \dots
  \cup\OO_{S_r}(\Phi_r^{N_r-1}(\alpha))\right).
\end{eqnarray*}
Each $S_i$ is generated by $(r-1)$ endomorphisms of $G$, and thus, by the
induction hypothesis, Lemma~\ref{reduction to finite maps} holds for the intersection of the subvariety $V$ with each
$S_i$-orbit. On the other hand, each $\Phi_i$ is finite on $H$, and
$\Phi_1^{N_1}\cdots\Phi_r^{N_r}(\alpha)\in H(\bC)$; so, replacing $G$ by $H$,
and replacing $V$ by $V\cap H$, we are done.
\end{proof}

From now on in the proof of Theorem~\ref{main} (and thus in this Section), we assume each $\Phi_i$ is finite on $G$.
Clearly, it suffices to prove Theorem~\ref{main} for irreducible subvarieties $V$.

Next we
prove another important reduction.
\begin{lemma}
\label{replace by a multiple}
It suffices to prove
Theorem~\ref{main} for the $S$-orbit of $m\alpha$, and
for the variety $mV$, where $m$ is any positive integer.
\end{lemma}

\begin{proof}
We will prove an (apparently) even stronger statement; we will show that for any coset of a subsemigroup of the form $\gamma + F\cap \N^r$ (where $\gamma\in\N^r$ and $F\subset\Z^r$ is a subgroup), the set of tuples $(n_1,\dots,n_r)\in \left(\gamma+ F\cap \N^r\right)$ for which $\Phi_1^{n_1}\cdots \Phi_r^{n_r}(\alpha)\in V(\bC)$ belongs to the class $\CC$. On the other hand, this apparent stronger statement is equivalent with the original one. Indeed, if we know that the set of tuples $(n_1,\dots,n_r)\in \N^r$ for which $\Phi_1^{n_1}\cdots \Phi_r^{n_r}(\alpha)\in V(\bC)$ belongs to the class $\CC$, then Proposition~\ref{intersection of cosets yields cosets} yields that also when we restrict to those tuples which are in $\gamma +F\cap\N^r$, we also get a set from the class $\CC$. So, from now on, we let $S$ be the set of all endomorphisms
$$\{\Phi_1^{n_1}\cdots \Phi_r^{n_r}\text{ : }(n_1,\dots,n_r)\in \left(\gamma + F\cap \N^r\right)\}.$$

We proceed the proof of Lemma~\ref{replace by a multiple} by induction on $\dim(V)+r$; the case $\dim(V)=r=0$ is obvious. Note also that if $\dim(V)=0$, then Theorem~\ref{main} was already proved in higher generality in Proposition~\ref{zero-dimensional}.   
On the other hand, if $r=0$, then the semigroup $S$ is trivial, and so the statement of Theorem~\ref{main} is obvious without any conditions. Furthermore, if $r=1$, then Theorem~\ref{main} is true in more generality, as shown in \cite{p-adic}.

Now, fix a positive integer $m$ and suppose that
Theorem~\ref{main} is true for the intersection of the
$S$-orbit of $m \alpha$ with the variety $mV$. 
So, we know that there exists a set $E_m\in \CC$ which contains all tuples $(n_1,\dots,n_r)\in \left(\gamma + F\cap \N^r\right)$ such that 
$$\Phi_1^{n_1}\cdots \Phi_r^{n_r}(m\alpha)\in mV.$$ 
For every $\Psi\in S$ such that $\Psi(\alpha)\in V$, we have
$\Psi(m\alpha)\in mV$ (because $\Psi$ is an algebraic group endomorphism). Therefore, if $E$ is the set of tuples $(n_1,\dots,n_r)\in \left(\gamma + F\cap \N^r\right)$ for which 
$$\Phi_1^{n_1}\cdots \Phi_r^{n_r}(\alpha)\in V,$$
then $E\subset E_m$. Let  $\beta + (H\cap \N^r)$, where $\beta\in\N^r$ and $H$ is a subgroup of $\Z^r$, be one of the cosets of subsemigroups contained in $E_m$; by definition of the class $\CC$, we know that $E_m$ is a finite union of such cosets of subsemigroups. It suffices to show that $E\cap \left(\beta+ (H\cap \N^r)\right)\in \CC$.

For $\beta$ and $H$ as above, we let 
$$S_1:=\{\Phi^{n_1}\cdots \Phi_r^{n_r}\text{ : }(n_1,\dots,n_r)\in  H\cap \N^r\},$$
and we let $\alpha_1:=\Phi_1^{b_1}\cdots \Phi_r^{b_r}(\alpha)$, where $\beta:=(b_1,\dots,b_r)$. Clearly, it suffices to show that the set $E_{\beta}$ of tuples $(n_1,\dots,n_r)\in H\cap \N^r$ for which
$$\Phi_1^{n_1}\cdots \Phi_r^{n_r}(\alpha_1)\in V(\bC)$$
belongs to the class $\CC$.

We let $V_1$ be the Zariski closure of $\OO_{S_1}(\alpha_1)$. If $V$ is not contained in $V_1$, then $V\cap V_1$ has smaller dimension than $V$ (since $V$ is irreducible), and because
$$V(\bC)\cap \OO_{S_1}(\alpha_1)= (V\cap V_1)(\bC)\cap \OO_{S_1}(\alpha_1),$$
we are done by the induction hypothesis (also note that if we are in part (b) of Theorem~\ref{main}, then $V\cap V_1$ is zero-dimensional, and thus the conclusion follows as shown in Proposition~\ref{zero-dimensional}). 

Assume now that $V\subset V_1$. Since, by construction, $m\cdot \OO_{S_1}(\alpha_1)\subset mV$ (because each $\Phi_i$ is an algebraic group endomorphism), we have that $mV_1\subset mV$. Therefore, $\dim(V)=\dim(V_1)$, and moreover $V$ is an irreducible component of $V_1$. For each $\sigma\in S_1$, we have $\sigma(V_1)\subset V_1$ since $\sigma(\OO_{S_1}(\alpha_1))\subset \OO_{S_1}(\alpha_1)$. Furthermore, since each $\sigma$ is a finite map, we conclude that $\sigma$ maps each irreducible component of $V_1$ of maximal dimension into another such component; hence $V$ is preperiodic under $\sigma$. 
There are two cases, which are covered in the following two propositions.

\begin{prop}
\label{V is not periodic}
If there exists $\sigma\in S_1$ such that $V$ is not periodic under $\sigma$, then $E_{\beta}\in \CC$.
\end{prop}

\begin{proof}
First of all, at the expense of replacing $\sigma$ by an iterate, we may assume that $V\not\subset\sigma(V_1)$ (since $\sigma$ maps the finite set of  irreducible components of $V_1$ of maximal dimension into itself, but no power of $\sigma$ preserves $V$).

We let $(k_1,\dots,k_r)\in \N^r$ such that $\sigma = \Phi_1^{k_1}\cdots \Phi_r^{k_r}$. We apply Proposition~\ref{reduction of r} to the element $(k_1,\dots,k_r)$ of the set $ H\cap \N^r$ and construct the subsets:
$$H_{(k_1,\dots,k_r)}:=\{\gamma\in H\cap \N^r\text{ : }(k_1,\dots,k_r)\prec \gamma\}$$
and 
$$H_{i,j}:=\{(\ell_1,\dots,\ell_r)\in  H\cap \N^r\text{ : }\ell_i=j\},$$
for $i=1,\dots,r$ and for $0\le j<k_i$. Then, as in Proposition~\ref{reduction of r}, we have 
$$H\cap\N^r=H_{(k_1,\dots,k_r)}\bigcup\left(\bigcup_{\substack{1\le i \le r\\ 0\le j \le k_i-1}}H_{i,j}\right).$$
It suffices to show that $E_{\beta}\cap H_{(k_1,\dots,k_r)}\in \CC$ and also that each $E_{\beta}\cap H_{i,j}\in \CC$.

Propositions \ref{reduction of r} and \ref{the cosets class is closed under projections} show that the projection of each $H_{i,j}$ on the $(r-1)$-st coordinates of $\N^r$ other than the $i$-th coordinate is also in the class $\CC$. Therefore the intersection $E_{\beta}\cap H_{i,j}$ belongs to the class $\CC$ by the inductive hypothesis, since we reduced the original problem to the intersection between the subvariety $V$ with the orbit of a point under a finitely generated commuting semigroup based on $(r-1)$ endomorphisms of $G$.

Now, for each $(\ell_1,\dots,\ell_r)\in H_{(k_1,\dots,k_r)}\cap E_{\beta}$, we actually have that
$$\Phi_1^{\ell_1}\cdots \Phi_r^{\ell_r}(\alpha_1)\in (V\cap \sigma(V_1))(\bC),$$
because $(k_1,\dots,k_r)\prec (\ell_1,\dots,\ell_r)$ and $\sigma=\Phi_1^{k_1}\cdots \Phi_r^{k_r}$.
Since $V$ is an irreducible variety and $V\not\subset \sigma(V_1)$, we conclude that $V\cap \sigma(V_1)$ has smaller dimension than $\dim(V)$. Thus $E_{\beta}\cap H_{(k_1,\dots,k_r)}\in \CC$ by our induction hypothesis, since we reduced the original problem to the intersection between the orbit of a point under a finitely generated commuting semigroup based on $r$ endomorphisms of $G$ with a subvariety of dimension less than $\dim(V)$. 
\end{proof}

\begin{prop}
\label{V is periodic}
If for each $\sigma\in S_1$, the variety $V$ is periodic under $\sigma$, then $E_{\beta}\in \CC$.
\end{prop}

\begin{proof}
According to Proposition~\ref{finitely generated semigroups}, there are finitely many generators $\{\tau_i\}_{1\le i\le s}$ for $H\cap \N^r$. For each $i=1,\dots,s$, we let $\sigma_i\in S_1$ be the element of $S_1$ corresponding to $\tau_i$; more precisely, if $\tau_i:=(k_1,\dots,k_r)$, then we let $\sigma_i:=\Phi_1^{k_1}\cdots \Phi_r^{k_r}$.

Using the hypothesis of Proposition~\ref{V is periodic}, for each $i=1,\dots,s$, we let $N_i\in\N$ such that $\sigma_i^{N_i}(V)=V$. We let $N$ be the least common multiple of all these $N_i$. Then $\sigma^N(V)=V$ for all $\sigma\in S_1$ (since $H\cap\N^r$ is generated by the $\tau_i$'s). In particular, we get that if $\gamma\in E_{\beta}$, and if $\tau \in   H\cap \N^r$, then $\gamma + N\cdot \tau\in E_{\beta}$. The conclusion of Proposition~\ref{V is periodic} follows from Proposition~\ref{important step semigroup}.   
\end{proof}
Propositions~\ref{V is not periodic} and \ref{V is periodic} finish the proof of Lemma~\ref{replace by a multiple}.
\end{proof}


\section{The p-adic uniformization}
\label{p-adic section}
We continue with the notation (and the reductions) from Section~\ref{reductions}. 

Let $\Gamma$ be the subgroup of $G(\bC)$ spanned by $\OO_S(\alpha)$
(see Lemma~\ref{simple lemma}). Theorem~\ref{T:F} yields that
$V(\bC)\cap\Gamma$ is a finite union of cosets of subgroups of
$\Gamma$. The Zariski closure of each coset of a subgroup of $\Gamma$
is a a translate of an algebraic subgroup of $G$. Hence, it suffices
to describe the intersection $\OO_S(\alpha)\cap G_1(\bC)$, where
$G_1=\beta+G_0$ is a translate of a connected algebraic subgroup $G_0$
of $G$.

Let $R$ be a finitely generated extension of $\Z$ over which $G$,
$G_0$, and each $\Phi_i$ can be defined.  By \cite[Proposition
4.4]{Jason} (see also \cite[Proposition 3.3]{p-adic}), there exists a
$p$ and an embedding of $R$ into $\Z_p$ such that
 
\begin{enumerate}
\item $G$ and $G_0$ have smooth models $\cG$ and $\cG_0$ over $\Z_p$;
\item each $\Phi_i$ extends to an unramified endomorphism of $\cG$;
\item $\alpha$ and $\beta$ extend to points in $\cG(\Z_p)$.
\end{enumerate}

Let $J_i$ denote the map $\Phi_i$ induces on the tangent space of
$0$.  By (2) and \cite[Proposition 2.2]{Jason}, if one chooses coordinates
for this tangent space via generators for the completed local ring at
0, then the entries of the matrix for $J_i$ will be in $\Z_p$.  Fix
one such set of coordinates and let $| \cdot |_p$ denote the corresponding $p$-adic
metric on the tangent space at 0. 

According to \cite[Proposition 3, p. 216]{NB} there exists a $p$-adic
analytic map $\exp$ which induces an analytic isomorphism between a
sufficiently small neighborhood $\cU$ of $\bG_a^g(\bC_p)$ and a
corresponding $p$-adic neighborhood of the origin $0\in
\cG(\bC_p)$. Furthermore (at the expense of possibly replacing $\cU$ by a smaller set), we may assume that the neighborhood $\cU$ is
a sufficiently small open ball, i.e., there exists a (sufficiently small) positive real number $\epsilon$ such that $\cU$ consists of all 
$(z_1,\dots,z_g)\in\bC_p^g$ satisfying $|z_i|_p<\epsilon$.  Because
$\exp(\cU)\cap \cG(\Z_p)$ is an open subgroup of the compact group
$\cG(\Z_p)$ (see \cite{p-adic}), we conclude that $\exp(\cU)\cap
\cG(\Z_p)$ has finite index $m$ in $\cG(\Z_p)$. So, at the expense of
replacing $\alpha$ by $m\alpha$, and $G_1$ by $mG_1=m\beta+mG_0$
(which is allowed by Lemma~\ref{replace by a multiple}) we may assume
that $\alpha,\beta\in \exp(\cU)$. Therefore, there exists $u_0,v_0\in
\cU$ such that
$
\exp(u_0)=\alpha$ and $\exp(v_0)=\beta.$
Furthermore, for each $i=1,\dots,r$, and for each $z\in\cU$ we have
$\exp(J_iz)=\Phi_i(\exp(z))$, since $\exp$ induces a local $p$-adic analytic isomorphism between $\bC_p^g$ and $\cG(\bC_p)$.
Finally, there exists a linear subvariety $\cV\subset \bG_a^g$ such that
$\exp(\cV(\cU))$ is a $p$-adic neighborhood of the origin $0\in \cG_0(\bC_p)$.

Thus we reduced Theorem~\ref{main} to proving the following result.
\begin{prop}
\label{everything in linear algebra}
In order to prove Theorem~\ref{main}, it suffices to show that the set $E$ of tuples $(n_1,\dots,n_r)\in \N^r$ for which $J_1^{n_1}\cdots J_r^{n_r}(u_0)\in (v_0+ \cV)$ belongs to the class $\CC$.
\end{prop}

\begin{proof}
Since the entries of each $J_i$ are in $\Z_p$, we have $J_i(\cU)\subset \cU$ for each
$i=1,\dots,r$ (also note the fact that $\cU$ is an open ball given by the norm inequalities $|z_i|_p<\epsilon$ for each coordinate $z_i$). So, for each $(n_1,\dots,n_r)\in \N^r$ we have $\Phi_1^{n_1}\cdots \Phi_r^{n_r}(\alpha)\in (\beta+\cG_0(\bC_p))$ if and only if
$$
J_1^{n_1}\dots J_r^{n_r} (u_0) \in v_0 +\cV,
$$
since $\Phi_1^{n_1}\cdots \Phi_r^{n_r}(\alpha)=\exp\left(J_1^{n_1}\cdots J_r^{n_r}(u_0)\right)$.
\end{proof}


\section{Proof of our main result}
\label{our main proofs}

We continue with the notation and reductions from the previous two Sections.

\subsection{Proof of Theorem~\ref{main}(a)}
Because the matrices $J_i$ are all diagonalizable and all commute,
they must be
simultaneously diagonalizable.  To see this, note that if $E_\lambda$
is a $\lambda$-eigenspace for a transformation $T$, and $U$ is a
transformation that commutes with $T$, then $U E_\lambda \subseteq
E_\lambda$ since for any $w$ in $E_\lambda$ we have 
$T(U w) = UTw = U(\lambda w) = \lambda U w$.  Hence the $J_i$ admit a common decomposition into eigenspaces and,
therefore, a common eigenbasis.

Let $C$ be an invertible matrix such that
$C^{-1}J_iC=B_i$ is diagonal for each $i=1,\dots,r$. Letting $u_1:=C^{-1}u_0$,
and $v_1:=C^{-1}v_0$, and $\cV_1:=C^{-1}\cV$, Proposition~\ref{everything in linear algebra} reduces to describing the set of tuples $(n_1,\dots,n_r)\in\N^r$ satisfying
\begin{equation}
\label{classical M-L}
B_1^{n_1}\dots B_r^{n_r}(u_1)\in v_1+\cV_1.
\end{equation}
Using Theorem~\ref{T:F} for the algebraic torus $\bG_m^g$ and its linear
subvariety $v_1+\cV_1$, we conclude that the set of tuples $(n_1,\dots,n_r)$
satisfying \eqref{classical M-L} lie in the intersection of $\N^r$ with a
finite union of cosets of subgroups of $\Z^r$. According to Proposition~\ref{everything in linear algebra}, this concludes the proof of
Theorem~\ref{main}(a) (see also Proposition~\ref{cosets of subgroups yield subsemigroups}).

\subsection{Proof of Theorem~\ref{main}(b)}
\label{curves}
 
In this case, in Proposition~\ref{everything in linear algebra} we know in addition that $v_0=0$ (since the subvariety $V$ is a connected algebraic group, not a translate of it), and furthermore, the linear subvariety $\cV$ is actually just a line $L$ passing through the origin (since the algebraic subgroup has dimension equal to one).

The following simple result will give rise to a proof of
Theorem~\ref{main}(b).  

\begin{prop}\label{group}
Let $V$ be a vector space over a field of characteristic $0$, let $T$ be a group of
invertible linear transformations of $V$, let $v \in V$,
let $Z$ be a one-dimensional subspace of $V$, and let 
$$ \cS = \{ g \in T \col g(v) \in Z \}.$$
If $\cS$ is nonempty, then $\cS$ is a right coset of a subgroup of $T$.   
\end{prop}
\begin{proof}
If $v$ is the zero vector, then the result is immediate. So, from now on, assume that $v$ is nonzero.

  Let $T_Z = \{ g \in T \col g(Z) = Z \}$ be the subgroup of $G$ which stabilizes the linear subspace $Z$.  If $\cS$ is not empty,
  we may choose some $a \in \cS$.  Then for any $b \in \cS$, we have
  $ba^{-1}(a(v)) \in Z$.  Since $a(v) \not= 0$ and $ba^{-1}$ is linear
  this means that $ba^{-1} \in T_Z$.  Thus, $b$ is in the coset $T_Z
  a$.  Similarly, $ga(v) \in Z$ for any $ga \in T_Z a$.  Thus, $\cS =
  T_Z a$.  
\end{proof}

Now we are ready to finish the proof of Theorem~\ref{main}(b).

    Let $T$ be the
  group of matrices generated by the $J_i$ (note that each $J_i$ is invertible since each $\Phi_i$ is a finite map by the reduction proved in Lemma~\ref{reduction to finite maps}).  Then, by
  Proposition~\ref{group}, the set of $g \in T$ such that $g(u_0) \in
  L$ is a coset $T_L a$ of a subgroup $T_L$ of $T$ which stabilizes the one dimensional variety $L$.  Now, for $i=1,\dots, r$,
  let $e_i$ be the element of $\Z^r$ whose $j$-th coordinate is the
  Kronecker delta $\delta_{ij}$.  Let $\rho$ be the map from $\Z^r$ to
  $T$ that takes $e_i$ to $J_i$. Then $\rho^{-1}(T_L a)$ is a coset $\gamma+H$ of a subgroup $H\subset \Z^r$. This shows that the set $E$ of tuples $(n_1,\dots,n_r)\in \N^r$ for which $J_1^{n_1}\cdots J_r^{n_r}(u_0)\in L$ is in fact $(\gamma+H)\cap \N^r$. Using Proposition~\ref{cosets of subgroups yield subsemigroups} we obtain that $E\in \CC$, as desired in the conclusion of Proposition~\ref{everything in linear algebra}. This concludes the proof of Theorem~\ref{main}(b).

\begin{remark}
\label{another remark}
  Although Proposition~\ref{group} holds for {\it any} group $T$ (not
  just finitely generated commutative groups $T$),
  Lemma~\ref{reduction to finite maps} only holds for finitely
  generated commutative semigroups $S$, so we are only able to reduce to
  Proposition~\ref{group} under the assumption that our semigroup is
  commutative and finitely generated.  On the other hand, if the
  semigroup is finitely generated and all of its elements are finite,
  then the proof of Theorem~\ref{main}(b) should go through, even when
  the semigroup is not commutative.  The case of semigroups that are
  not finitely generated is more troublesome; for example, it may not
  be possible to find a $p$ such that all of the generators are
  defined over $\Z_p$. 
\end{remark}

\subsection{Proof of Theorem 1.3(c)}
\label{G_m^2}
We already proved Theorem~\ref{main}(c) when $G=A^j$ for $j<2$ in Corollary~\ref{j=0,1} (note that a one-dimensional semiabelian variety is clearly simple). Therefore, from now on, we may assume $G=A^2$.

Since it is a one-dimensional semiabelian variety, $A$ is either isomorphic to $\bG_m$ or to an elliptic curve $E$. In either case, $\End(E)$ is isomorphic with a subring $\cR$ of the ring $\ok$ of algebraic integers in a given quadratic imaginary field $K$. 
We already proved the result if (for each $i$) the Jacobian $J_i$ of $\Phi_i$
at the identity of $A^2$ is diagonalizable. Therefore, we may assume there
exists at least one such Jacobian which is not diagonalizable.
Since the endomorphism ring of $A^2$ is isomorphic with the ring $M_{2,2}(\cR)$ of $2$-by-$2$ matrices with integer coefficients, we have that each $J_i\in M_{2,2}(\cR)$. 

The following result is an elementary application of linear algebra.
\begin{prop}
\label{linear algebra}
Let $J_1,\dots, J_r\in M_{2,2}(\cR)$ be commuting $2$-by-$2$ matrices. Assume there exists
$i=1,\dots, r$ such that $J_i$ is not diagonalizable. Then there exists an
invertible $2$-by-$2$ matrix $B$ defined over $K$ such that for each
$i=1,\dots,r$ there exist $a_i \in \ok$ and $c_i \in K$ such that
\[BJ_iB^{-1} = a_i \cdot \left(\begin{array}{cc}
1 & c_i\\
0 & 1
\end{array}\right).\]
\end{prop}
\begin{proof}
Note in general that if $\lambda_1 \not= \lambda_2$, then
$\left(\begin{array}{cc}
 \lambda_1 & 0\\
 0 &  \lambda_2
 \end{array}\right)$ 
commutes with
$ \left(\begin{array}{cc}
a & b\\
c & d 
\end{array}\right)$
if and only if $c = b = 0$.  Thus, if there is some $J_\ell$ that has
distinct eigenvalues, then all of the $J_i$ must be diagonalizable,
which we have assumed is not the case. Thus, each $J_i$ must have only
one eigenvalue, call it $\lambda_i$.  Note that since each $J_i$ has
entries in $\ok$ and a repeated eigenvalue $\lambda_i$, then for each
$i$, we have that $\lambda_i$ is the only root of a second degree
monic polynomial with coefficients in $\ok$. Hence each $\lambda_i$ is
an algebraic integer, and moreover, it belongs to the field $K$; thus
$\lambda_i\in\ok$ for each $i$.  Now, if two invertible
non-diagonalizable $2$-by-$2$ matrices $M$ and $N$ commute, then any
eigenvector for $M$ must also be an eigenvector for $N$.  Thus, there
exists a change of basis matrix $B$ with entries in $K$ such that
$BJ_iB^{-1} = a_i \cdot \left(\begin{array}{cc}
1 & c_i\\
0 & 1
\end{array}\right)$
for each $i$ (where $a_i=\lambda_i\in \ok$ and $c_i\in K$).  
\end{proof}
Note that each $a_i$ is nonzero since we assumed that each $\Phi_i$ is an isogeny (see Lemma~\ref{reduction to finite maps}).

Since we work in a two-dimensional space, the linear variety $\cV$ from Proposition~\ref{everything in linear algebra} is a line $L$ through the origin (note that in Proposition~\ref{zero-dimensional} we already covered the case $\dim(V)=0$, while the case $V=A^2$ is trivial). So, if we let
$\left(\begin{array}{c}
\delta_1\\
\delta_2
\end{array}\right)$
for some constants $\delta_1,\delta_2\in\bC_p$, then
Proposition~\ref{everything in linear algebra} reduces the problem to determining for which tuples $(n_1,\dots,n_r)\in \N^r$ we have that
\[\prod_{i=1}^r a_i^{n_i} \cdot \left(\begin{array}{cc}
1 & \sum_{i=1}^r n_i c_i\\
0 & 1
\end{array}\right)\cdot \left(\begin{array}{c}
\delta_1\\
\delta_2
\end{array}\right) \in v_0 + L.
\]
So, there exist constants $\delta_3,\delta_4,\delta_5\in\bC_p$ such that the above equation reduces to
\begin{equation}
\label{ugly equation}
\prod_{i=1}^r a_i^{n_i}\cdot \left(\delta_3\cdot \left(\delta_1 + \delta_2 \cdot \left(\sum_{i=1}^r n_i c_i\right)\right)+\delta_4\cdot \delta_2\right)=\delta_5.
\end{equation}
After collecting the terms, and introducing new constants $A,B\in \bC_p$, we conclude that the equation \eqref{ugly equation} reduces to either
\begin{equation}
\label{not weird equation}
\sum_{i=1}^r c_i n_i = B,
\end{equation}
or
\begin{equation}
\label{not so weird equation}
\prod_{i=1}^r a_i^{n_i} = A,
\end{equation}
or
\begin{equation}
\label{weird equation}
\prod_{i=1}^r a_i^{n_i} \cdot \left(A+\sum_{i=1}^r n_i c_i \right)=B.
\end{equation} 
It is immediate to see that the set $E$ of solution tuples $(n_1,\dots,n_r)\in \N^r$ to equation \eqref{not weird equation} belongs to the class $\CC$. Indeed, $E$ is a finite union of cosets of $H\cap \N^r$, where $H\subset \Z^r$ is the subgroup consisting of all solutions $(n_1,\dots,n_r)\in\Z^r$ to the equation $\sum_{i=1}^r c_i n_i = 0$; the cosets are translates of $H$ by the minimal elements of $E$ with respect to the usual order $\prec$ on $\N^r$ (see Proposition~\ref{finitely many minimal elements}).

Either arguing directly (note that each $|a_i|^2\in\N$ since $a_i$ is a nonzero algebraic integer from a quadratic imaginary field), or alternatively using the proof of the Mordell-Lang Conjecture for the multiplicative group \cite{Lang_int} we conclude that the set of tuples $(n_1,\dots,n_r)$ which are solutions to equation \eqref{not so weird equation} also belongs to the class $\CC$.

Now, assume equation \eqref{ugly equation} reduces to \eqref{weird equation}. Also note that we may assume that $B\ne 0$ since otherwise we would have that $\sum_{i=1}^r c_i n_i = -A$ (note that each $a_i$ is nonzero), which reduces everything to the equation \eqref{not weird equation} studied above.

We will show first that it is impossible that the set of values for $\prod_{i=1}^r a_i^{n_i}$ is infinite as $(n_1,\dots,n_r)$ runs through the tuples of all solutions to equation \eqref{weird equation}. Indeed, if the set of values for $\prod_{i=1}^r a_i^{n_i}$ would be infinite, then we obtain that $A\in K$, and thus, that $B\in K$ (since $a_i\in \ok$ and $c_i\in K$ for each $i$). Furthermore, at the expense of replacing both $A$ and $B$, but also the $c_i$'s by integral multiples of them, we may even assume that $A,B\in \ok$ and also $c_i\in\ok$ for each $i$. 
Since $B\ne 0$, then $A+\sum_{i=1}^r c_in_i\ne 0$. On the other hand, $A+\sum_{i=1}^r c_in_i\in \ok$, and $\ok$ is the ring of integers in a quadratic imaginary field. Therefore, the absolute value
$\mid A+\sum_{i=1}^r c_i n_i\mid$ is uniformly bounded away from $0$; more precisely, $\mid A+\sum_{i=1}^r c_i n_i \mid\ge 1$. 
So, for each solution tuple $(n_1,\dots,n_r)$ to \eqref{weird equation}, we have that $\prod_{i=1}^r |a_i|^{n_i} \le |B|$.
On the other hand, because $\ok$ is the ring of integers in a quadratic imaginary field, there are finitely many elements $D\in\ok$ such that $|D|\le |B|$. This contradicts our assumption that the set of values for $\prod_{i=1}^r a_i^{n_i}$ is infinite.

Now, since there are only finitely many distinct values for
$\prod_{i=1}^r a_i^{n_i}$ as $(n_1,\dots,n_r)$ runs over all tuples
which are solutions to equation \eqref{weird equation}, then
\eqref{weird equation} reduces to solving simultaneously finitely many
pairs of equations of the form $\prod_{i=1}^r a_i^{n_i} = D$ and
$\sum_{i=1}^r c_i n_i = - A + B/D$ for some numbers $D$ belonging to a
finite set. However, the solutions set of tuples
$(n_1,\dots,n_r)\in\N^r$ to either one of the above equations
belongs to the class $\CC$ (as we argued before for equations
\eqref{not weird equation} and \eqref{not so weird equation}). Since
the class $\CC$ is closed under intersections (see
Proposition~\ref{intersection of cosets yields cosets}), we are done.


\section{Counterexamples}
\label{counterexamples}

In this section we give two examples in which
Question~\ref{semigroup semiabelian}
has a negative answer.  Both examples involve two algebraic group endomorphisms
of $\bG_m^3$.  In the first example, the subvariety $V$ is a line which is a
coset of a one-dimensional algebraic subgroup of $\bG_m^3$.  In this example
the endomorphisms are in fact automorphisms.  In the second example, the
subvariety $V$ is a two-dimensional algebraic subgroup of $\bG_m^3$, and
the algebraic group endomorphisms do not induce automorphisms on any
positive-dimensional subvariety of $\bG_m^3$ (see Lemma~\ref{not automorphisms}).

\subsection{An example where $V$ is a line.}\label{ex1}
Let $\Phi$ and $\Psi$ be the endomorphisms of $\bG_m^3$ given by
$$
\Phi(x,y,z) = \left(xy^{-1},\, yz^{-2},\, z\right) 
\quad \text{and} \quad
\Psi(x,y,z) = \left(xy^2,\, yz^4,\, z\right).
$$
Then $\Phi\Psi(x,y,z)=\Psi\Phi(x,y,z)=\left(xyz^{-4},\, yz^2,\, z\right)$.
Moreover, for any $m,n\in\N$ we have
\[
\Phi^m\Psi^n(x,y,z)=\left(xy^{2n-m}z^{4n(n-1)-4mn+m(m-1)},\,yz^{4n-2m},\,
z\right).
\]

Let $V$ be the subvariety $\{(2,y,3)\text{ : }y\in\bC\}$ of $\bG_m^3$,
and let $\alpha:=(2,9,3)$.  Then, for $m,n\in\N$, the point
$\Phi^m\Psi^n(\alpha)$ lies in $V(\bC)$ if and only if
\[
2(2n-m) + 4n(n-1) - 4mn + m(m-1) = 0,
\]
or, in other words,
$(2n-m)^2 = 3m$. The solutions of this last equation in nonnegative integers are visibly
$m = 3k^2$ and $n = 3(k^2\pm k)/2$ with $k\in\N$.  But the set $
\{(3k^2,3(k^2\pm k)/2) : k \in \N\}$
is not in the class $\CC$ (for instance, because there are arbitrarily large gaps
between consecutive values of the first coordinate).

In this example, $V$ is a line, and in fact is the translate of the
one-dimensional algebraic subgroup
$\{(1,y,1)\text{ : }y\in\bC\}$ by the point $(2,1,3)$.
This example resembles previous counterexamples to our general  Question~\ref{general semigroup question}, in that
$\Phi$ and $\Psi$ are automorphisms.  Below we present another example which
does not have this property.

\subsection{An example involving non-automorphisms.}\label{ex2}
Let $\Phi$ and $\Psi$ be the endomorphisms of $\bG_m^3$ given by

$$\Phi(x,y,z) = \left(x^2y^{-1},\, y^2z^{-2},\, z^2\right) \\
\quad \text{and} \quad
\Psi(x,y,z) = \left(x^2y^2,\, y^2z^4,\, z^2\right).$$
Then $\Phi\Psi(x,y,z)=\Psi\Phi(x,y,z)=
\left(x^4y^2z^{-4},\, y^4z^4,\, z^4\right)$.
Moreover, for any $m,n\in\N$ we have that $\Phi^m\Psi^n(x,y,z)$
equals
\begin{align*}
\bigl(x^{2^{m+n}} &y^{n 2^{m+n} - m 2^{m+n-1}}
   z^{n(n-1) 2^{m+n} + m(m-1) 2^{m+n-2} - mn 2^{m+n}},\\
&y^{2^{m+n}} z^{n 2^{m+n+1} - m 2^{m+n}},\, z^{2^{m+n}}\bigr).
\end{align*}

Let $V$ be the subvariety of $\bG_m^3$ defined by $x=1$, and let
$\alpha:=(1,1/3,9)$.  Then, for $m,n\in\N$, the point $\Phi^m\Psi^n(\alpha)$
lies in $V(\bC)$ if and only if
\[
\Bigl(\frac13\Bigr)^{n 2^{m+n} - m 2^{m+n-1}}
  9^{n(n-1) 2^{m+n} + m(m-1) 2^{m+n-2} - mn 2^{m+n}} = 1,
\]
or in other words
\[
2^{m+n-1}
(- 2n + m + 4n(n-1) + m(m-1) - 4mn) = 0.
\]
This last equation is equivalent to $(2n - m)^2 = 6n$, whose solutions in
nonnegative integers are visibly $n=6k^2$ and $m = 12k^2 \pm 6k$ with $k\in\N$.
But the set of solutions $\{(12k^2\pm 6k,\,6k^2): k\in\N\}$ is not in $\CC$ (for instance, because there are arbitrarily large gaps between consecutive
values of the second coordinate).

This example differs from all previous counterexamples to the generalization
of Question~\ref{semigroup semiabelian} in that the semigroup $S$ generated
by $\Phi$ and $\Psi$ contains no nonidentity element having a degree-one
restriction to some positive-dimensional subvariety of $\bG_m^3$.
The reason for this is that $\Phi^m\Psi^n$ has characteristic polynomial
$(X-2^{m+n})^3$ (see Lemma~\ref{not automorphisms} for a more general result).

\begin{lemma}
\label{not automorphisms}
Let $g\in\N$, and let $\Phi\in\End(\bG_m^g)$ be an algebraic group endomorphism whose Jacobian at the identity of $\bG_m^g$ has all eigenvalues of absolute value larger than one. Then there is no positive-dimensional subvariety $V$ of $\bG_m^g$ such that $\Phi$ restricts to an automorphism of $V$.
\end{lemma}

\begin{proof}
Assume there exists a positive-dimensional subvariety $V\subset \bG_m^g$ such that $\Phi$ restricts to an automorphism of $V$. Let $\alpha\in V(\bC)$ be a nontorsion point of $\bG_m^g$ (note that $V$ is positive-dimensional, and thus not all its points are torsion). Because $\Phi$ restricts to an endomorphism of $V$, the $\Phi$-orbit $\OO_{\langle\Phi\rangle}(\alpha)$ of $\alpha$ is contained in $V$. Applying Lemma~\ref{simple lemma}, there exists a finitely generated subgroup $\Gamma\subset \bG_m^g(\bC)$ such that $\OO_{\langle\Phi\rangle}(\alpha)$ is contained in $\Gamma$. Because $\OO_{\langle\Phi\rangle}(\alpha)$ is an infinite set contained in the intersection $V(\bC)\cap\Gamma$, we obtain that the Zariski closure $W$ of $V(\bC)\cap\Gamma$ is a positive-dimensional subvariety of $V$. Using Theorem~\ref{T:F}, we conclude that $W$ contains a translate of a positive-dimensional algebraic subgroup $H$ of $\bG_m^g$. Because $\Phi$ restricts to an automorphism of $V$, then it induces a degree-one map on $H$. Therefore, the Jacobian $J_{\Phi,H}$ of $\Phi$ at the identity of $H$ (which is also the identity of $\bG_m^g$) is an invertible matrix over $\Z$; however, this is contradicted by the fact that all eigenvalues of $J_{\Phi,H}$ have absolute value larger than one (because they are among the eigenvalues for the Jacobian of $\Phi$ at the identity of $\bG_m^g$).
\end{proof}

\begin{remark}
  Using the inverse of the $p$-adic exponential map (as in the proof of Theorem~\ref{main}(a)) reduces
  the examples of \ref{ex1} and \ref{ex2} to simple linear algebra.
  Matrices that commute must share a common Jordan block
  decomposition, which is slightly weaker than being simultaneously
  diagonalizable.  So, for example, applying the $p$-adic logarithmic map to the endomorphisms in \ref{ex1}
  converts $\Phi$ into $\left( \begin{array}{rrr} 1 & -1 & 0 \\ 0 & 1
      & -2 \\ 0 & 0 & 1 \end{array} \right)$ and converts $\Psi$ into
  $\left( \begin{array}{rrr} 1 & 2 & 0 \\ 0 & 1 & 4 \\ 0 & 0 &
      1\end{array} \right)$.  Note also that the above counterexamples can be also constructed for endomorphisms of $E^3$, where $E$ is an elliptic curve, since the whole argument is at the level of linear algebra after applying the $p$-adic logarithmic maps.
\end{remark}
 



\begin{thebibliography}{9}
\newcommand{\au}[1]{{#1},}
\newcommand{\ti}[1]{\textit{#1},}
\newcommand{\jo}[1]{{#1}}
\newcommand{\vo}[1]{\textbf{#1}}
\newcommand{\yr}[1]{(#1),}
\newcommand{\pp}[1]{#1.}
\newcommand{\ppx}[1]{#1,}
\newcommand{\pps}[1]{#1;}
\newcommand{\bk}[1]{{#1},}
\newcommand{\inbk}[1]{in: {#1}}
\newcommand{\xxx}[1]{{arXiv:#1.}}

\bibitem{Bell}
\au{J.~P. Bell}
\ti{A generalised Skolem--Mahler--Lech theorem for affine varieties}
\jo{J. London Math. Soc. (2)}
\vo{73}
\yr{2006}
367--379; corrig. J. London Math. Soc. (2) \textbf{78} (2008), 267--272.
\xxx{math/0501309}

\bibitem{Jason}
\au{J.~ Bell, D.~ Ghioca, and T.~J.~ Tucker}
\ti{The dynamical Mordell--Lang problem for \`etale maps}
\jo{Amer. J. Math.}
\vo{132}
\yr{2010}
\pp{1655--1675}
\xxx{0808.3266}

\bibitem{PR}
\au{R.~L. Benedetto, D. Ghioca, T.~J. Tucker and P. Kurlberg (with an Appendix by U. Zannier)}
\ti{A case of the dynamical Mordell--Lang conjecture}
to appear in Math. Ann.,
\xxx{0712.2344}

\bibitem{NB}
\au{N.~Bourbaki}
\bk{Lie Groups and Lie Algebras. Chapters 1--3}
Springer--Verlag, Berlin, 1998.

\bibitem{Denis}
\au{L. Denis}
\ti{G\'eom\'etrie et suites r\'ecurrentes}
\jo{Bull. Soc. Math. France}
\vo{122}
\yr{1994}
\pp{13--27}

\bibitem{Faltings}
\au{G.~Faltings}
\ti{The general case of S. Lang's theorem}
\inbk{Barsotti symposium in Algebraic Geometry}
\ppx{175--182}
Academic Press, San Diego, 1994.

\bibitem{p-adic}
\au{D. Ghioca and T.~J. Tucker}
\ti{Periodic points, linearizing maps, and the dynamical Mordell--Lang problem}
\jo{J. Number Theory},
\vo{129}
\yr{2009}
\ppx{1392--1403}
\xxx{0805.1560}

\bibitem{lines}
\au{D. Ghioca, T.~J. Tucker, and M.~E. Zieve}
\ti{Intersections of polynomial orbits, and a dynamical Mordell--Lang theorem}
\jo{Invent. Math.}
\vo{171}
\yr{2008}
\ppx{463--483}
\xxx{0705.1954}

\bibitem{lines-2}
\au{D. Ghioca, T.~J. Tucker, and M.~E. Zieve}
\ti{Linear relations between polynomial orbits}
submitted for publication,
\xxx{0807.3576}

\bibitem{iitaka}
\au{S. Iitaka}
\ti{Logarithmic forms of algebraic varieties}
\jo{J. Fac. Sci. Univ. Tokyo Sect. IA Math.}
\vo{23}
\yr{1976}
\pp{525--544}

\bibitem{iitaka2}
\au{\bysame}
\ti{On logarithmic Kodaira dimension of algebraic varieties}
\inbk{Complex Analysis and Algebraic Geometry}
Iwanami Shoten, Tokyo
\yr{1977}
\pp{175--189}




\bibitem{Lang_int}
\au{S. Lang}
\ti{Integral points on curves}
\jo{Publ. Math. IHES}
\vo{6}
\yr{1960}
\pp{27--43}

\bibitem{Lech}
\au{C. Lech}
\ti{A note on recurring series}
\jo{Ark. Mat.}
\vo{2}
\yr{1953}
\pp{417--421}

\bibitem{Mahler}
\au{K. Mahler}
\ti{Eine arithmetische Eigenshaft der Taylor-Koeffizienten
rationaler Funktionen}
\jo{Proc. Kon. Ned. Akad. Wetensch.}
\vo{38}
\yr{1935}
\pp{50--60}

\bibitem{McQuillan}
\au{M. McQuillan}
\ti{Division points on semi-abelian varieties}
\jo{Invent. Math.}
\vo{120}
\yr{1995}
\pp{143--159}

\bibitem{Milne}
\au{J. Milne}
\ti{Abelian varieties}
available at www.jmilne.org/math/index.html.

\bibitem{Skolem}
\au{T. Skolem}
\ti{Ein Verfahren zur Behandlung gewisser exponentialer Gleichungen
und diophantischer Gleichungen}
\inbk{Comptes rendus du 8e congr\`es des math\'ematiciens scandinaves}
\yr{1935}
\pp{163--188}

\bibitem{V1}
\au{P. Vojta}
\ti{Integral points on subvarieties of semiabelian varieties. I}
\jo{Invent. Math.}
\vo{126}
\yr{1996}
\pp{133--181}

\end{thebibliography}
\end{document}